\documentclass[a4paper,11pt]{amsart}
\usepackage{a4wide}
\usepackage[utf8]{inputenc}
\usepackage{url}
\usepackage{pst-node}
\usepackage{cite}
\usepackage{hyperref}
\IfFileExists{\jobname.TEX}{}{\usepackage{srcltx}}
\usepackage{array}

\usepackage{color}

\newcommand{\field}[1]{\mathbb{{#1}}}
\newcommand{\ideal}[1]{\mathfrak{{#1}}}

\newcommand{\eps}{\varepsilon}

\newcommand{\C}{\field{C}}
\renewcommand{\L}{\field{{L}}}
\newcommand{\MC}{\mathcal{M}_C}

\newcommand{\Q}{\field{{Q}}}
\newcommand{\K}{\field{{K}}}
\newcommand{\E}{\field{{E}}}

\newcommand{\disc}{\Delta}
\newcommand{\DK}{{\disc_\K}}
\newcommand{\DL}{{\disc_\L}}

\newcommand{\DLK}{{\disc_{\L/\K}}}
\newcommand{\lDL}{{\log\DL}}
\newcommand{\llDL}{{\log\lDL}}
\newcommand{\nE}{{n_\E}}
\newcommand{\nK}{{n_\K}}
\newcommand{\nL}{{n_\L}}
\newcommand{\OO}{\mathcal{O}}
\newcommand{\OK}{{\OO_\K}}
\newcommand{\p}{{\ideal{p}}}

\renewcommand{\P}{{\ideal{P}}}
\newcommand{\Sg}{{\mathbf{S}}}
\newcommand{\iI}{{\ideal{I}}}
\newcommand{\Artin}[2]{{\genfrac[]{}{}{{#1}}{{#2}}}}
\newcommand{\intpart}[1]{\left\lfloor#1\right\rfloor}

\newcommand{\dd}{\,\mathrm{d}}

\newcommand{\Norm}{\textrm{\upshape N}}
\DeclareMathOperator{\Gal}{Gal}

\numberwithin{equation}{section}

\newtheorem*{theorem*}{Theorem}

\newtheorem{lemma}{Lemma}[section]
\newtheorem*{lemma*}{Lemma}

\theoremstyle{remark}

\newtheorem*{remark*}{Remark}

\newtheorem*{acknowledgements}{Acknowledgements}

\allowdisplaybreaks[4]

\begin{document}
\title[Conditional upper bound for the $k$-th prime ideal with given Artin symbol]
      {Conditional upper bound for the $k$-th prime ideal with given Artin symbol}

\author[L.~Greni\'{e}]{Lo\"{\i}c Greni\'{e}}
\address[L.~Greni\'{e}]{Dipartimento di Ingegneria Gestionale, dell'Informazione e della Produzione\\
         Universit\`{a} di Bergamo\\
         viale Marconi 5\\
         24044 Dalmine
         Italy}
\email{loic.grenie@gmail.com}

\author[G.~Molteni]{Giuseppe Molteni}
\address[G.~Molteni]{Dipartimento di Matematica\\
         Universit\`{a} di Milano\\
         via Saldini 50\\
         I-20133 Milano\\
         Italy}
\email{giuseppe.molteni1@unimi.it}

\keywords{}
\subjclass[2010]{Primary 11R42, Secondary 11Y70}

\date{\today.
}

\begin{abstract}
We prove an explicit upper bound for the $k$-th prime ideal with fixed Artin symbol, under the assumption
of the validity of the Riemann hypothesis for the Dedekind zeta functions.
\end{abstract}

\maketitle

\section{Introduction}\label{sec:1}
We recall some definitions, just to fix the notations. Let $\K$ be a number field, let $\nK$ denote its
dimension, $\DK$ the absolute value of its discriminant, and $r_1(\K)$, $r_2(\K)$ the number of its real
and complex places, respectively. The von Mangoldt function $\Lambda_\K$ is defined on the set of ideals
of $\OK$ as $\Lambda_\K(\iI) := \log\Norm\p$ if $\iI=\p^m$ for some $\p$ and $m\geq 1$, and is zero
otherwise, where $\p$ denotes any nonzero prime ideal and $\Norm\p$ its absolute norm.\\
Moreover, let $\K\subseteq \L$ be a Galois extension of number fields with relative discriminant $\DLK$.
For $\P$ a prime ideal of $\L$ above a non-ramified $\p$ of $\K$, the Artin symbol $\Artin{\L/\K}{\P}$
denotes the Frobenius automorphism corresponding to $\P/\p$, and $\Artin{\L/\K}{\p}$ the conjugacy class
of all the $\Artin{\L/\K}{\P}$. The symbol $\Artin{\L/\K}{.}$ is then extended multiplicatively to the
group of fractional ideals of $\K$ coprime to $\DLK$.\\
Finally, let $C$ be any conjugacy class in $G:=\Gal(\L/\K)$ and let $\eps_C$ be its characteristic
function. Then the function $\pi_C$ and the Chebyshev function $\psi_C$ are defined as
\begin{align*}
\pi_C(x)
&:= \sharp \Big\{\p\colon \p\text{ non-ramified in }\L/\K,\Norm\p \leq x,\Artin{\L/\K}{\p}=C\Big\}       \\
&\phantom{:}= \sum_{\substack{\p\\\p\text{ non-ram.}\\\Norm\p\leq x}}\eps_C\Big(\Artin{\L/\K}{\p}\Big),  \\
\psi_C(x)
&:= \sum_{\substack{\iI\subset\OK\\\iI\text{ non-ram.}\\\Norm\iI\leq x}}\eps_C\Big(\Artin{\L/\K}{\iI}\Big)\Lambda_\K(\iI).
\end{align*}
In~\cite{GrenieMolteni9} we have proved the following explicit bound.
\begin{theorem*}
Assume GRH holds. Let $x\geq 1$, then
\begin{equation}\label{eq:1.0A}
\Big|\frac{|G|}{|C|}\psi_C(x)-x\Big|
\leq \sqrt{x}\Big[\Big(\frac{\log   x}{2\pi}+2\Big)\lDL
                + \Big(\frac{\log^2 x}{8\pi}+2\Big)\nL
             \Big].
\end{equation}
\end{theorem*}
\noindent
This result concludes a quite long set of similar but partial computations, originated with Jeffrey
Lagarias and Andrew Odlyzko's paper~\cite{LagariasOdlyzko} where this result is proved with undetermined
constants, and which was followed by the result announced by Joseph Oesterl\'{e}~\cite{Oesterle} and the one
of Bruno Winckler~\cite[Th.~8.1]{Winckler} (both with the same generality and explicit but larger
constants), the one of Lowell Schoenfeld~\cite{Schoenfeld1} (same bound but only for the case
$\L=\K=\Q$), and our recent paper~\cite{GrenieMolteni3} (same conclusion, but only for the case $\L=\K$).
\medskip

Bound~\eqref{eq:1.0A} implies that for every class $C$ there is a prime ideal $\p$ with
\begin{equation}\label{eq:1.1A}
\Norm\p \leq \Big(\Big(\frac{1}{2\pi\log\delta}+o(1)\Big)\lDL(\llDL)^2\Big)^2
\end{equation}
which is not ramified and for which $\Artin{\L/\K}{\p}=C$, where $\delta$ is any lower bound for the root
discriminant of the family of fields for which we are interested to apply the result: $\sqrt{3}$ is a
possible value for all fields, and $\sqrt[3]{23}$ is another possible value when the
six quadratic fields $\Q[\sqrt{d}]$ with $d\in\{-7,-3,-2,-1,2,5\}$ are excluded.\\
This consequence of any bound similar to~\eqref{eq:1.0A} is already discussed in Lagarias and Odlyzko's
paper, where in fact the existence of a bound of the form $c(\lDL(\llDL)^2)^2$ for some computable (but
not explicit) constant $c$ is proved.

The appearance of the factor $(\llDL)^4$ in~\eqref{eq:1.1A} is a consequence of the use
of~\eqref{eq:1.0A}, which actually is not designed for that purpose. In fact, in essence, this comes down
to the fact that the kernel $\frac{x^s}{s}$, needed to relate $\psi_C$ to a convenient sum of logarithmic
derivatives of Artin $L$-functions, does not decay very quickly along the vertical lines.
To overcome this problem, the authors of~\cite{LagariasOdlyzko} also sketched a different approach
replacing $\frac{x^s}{s}$ with the kernel $\big(\frac{y^{s-1}-x^{s-1}}{s-1}\big)^2$. With a suitable
choice of the parameters $x$ and $y$ in terms of $\lDL$, this kernel allows to remove the factor
$(\llDL)^4$ from the bound.\\
This improvement is not exclusive of this specific kernel, and the same conclusion may be achieved also
via different kernels, provided that they decay quickly enough along vertical lines. In particular, we
have obtained~\eqref{eq:1.0A} as a by-product of computations for $\psi_C^{(1)}(x) := \sum_{n\leq
x}\psi_C(n)$, which is related to the kernel $\frac{x^{s+1}}{s(s+1)}$. Its decay along vertical lines is
better than the one of $\frac{x^s}{s}$ and is actually strong enough to get a bound of the type
of~\eqref{eq:1.1A}, without the factor $(\llDL)^4$.
In fact, we prove here the following claim as a consequence of some of the inner results we got
in~\cite{GrenieMolteni9}.
\begin{theorem*}
Assume GRH holds. Fix any class $C$ and any integer $k\geq 0$. Assume
\[
\sqrt{x}
\geq 1.075\lDL + \sqrt{2\tfrac{|G|}{|C|}k\log\big(\tfrac{|G|}{|C|}k\big)} + 2\tfrac{|G|}{|C|} + 15,
\]
where $k\log k$ is set to $0$ for $k=0$. Then $\pi_C(x) \geq k+1$. In other words, if we order the prime
ideals $\{\p_{k}\}_{k=1}^\infty$ which are not ramified and for which $\Artin{\L/\K}{\p}=C$ according to
their norm, for every $k\geq 0$ we have
\[
\Norm\p_{k+1} \leq \Big(1.075\lDL + \sqrt{2\tfrac{|G|}{|C|}k\log\big(\tfrac{|G|}{|C|}k\big)} + 2\tfrac{|G|}{|C|} + 15\Big)^2.
\]
\end{theorem*}
\noindent %
Thus, for instance, there is a non-ramified prime ideal $\p$ in $C$ with $\Norm\p \leq \bigl(1.075\lDL +
2\tfrac{|G|}{|C|} + 15\bigr)^2$ (case $k=0$) and two such ideals within $\Bigl(1.075\lDL +
\sqrt{2\tfrac{|G|}{|C|} \log\big(\tfrac{|G|}{|C|}\big)} + 2\tfrac{|G|}{|C|} + 15\Bigr)^2$ (case $k=1$).

The proof of this theorem shows that the constant $15$ can be removed when the degree of the field is
large enough, but the main constant $1.075$ is rooted in the method and can be improved only marginally.
In particular it remains larger than $1$. This implies that the case $k = 0$ of the theorem is weaker
than the analogous conclusion of the paper by Eric Bach and Jonathan
Sorenson~\cite[Th.~3.1]{BachSorenson}, further improved for the case where $\K=\Q$ and $\L/\Q$ is abelian
by Youness Lamzouri, Xiannan Li and Kannan Soundararajan~\cite[Th.~1.2]{LamzouriLiSoundararajan} (see
also~\cite{LamzouriLiSoundararajan2}).\\
The claim giving more ideals (i.e., $k\geq 1$) cannot be reached with Lagarias--Odlyzko's,
Bach--Sorenson's or Lam\-zou\-ri--Li--Soun\-da\-ra\-ra\-jan's approaches.
%
\smallskip

The case where $\K=\Q$ and $C$ is the trivial class has been considered also in~\cite[Corollary
2.1]{GrenieMolteni2}, with similar conclusions, in particular with the same constant for $\lDL$ but a
larger one for the $k\log k$ term.
%
%
\smallskip

Finally, we notice that if the field extension $\L/\K$ and the class $C$ are fixed and only the
dependence on $k$ is retained, then the theorem says that the norm of the $k$th prime ideals in $C$ is
$\leq (2 + o(1))\tfrac{|G|}{|C|}k\log k$: this is the correct upper bound in its dependency on $k$ and on
the density factor $\tfrac{|G|}{|C|}$, but we know from the prime ideals density theorems that the
absolute constant $2$ could be $1$. This overestimation represents the price we pay in order to get a
uniform and totally explicit result.

\begin{acknowledgements}
The authors are members of the INdAM group GNSAGA.
\end{acknowledgements}

\section{Preliminary facts}\label{sec:2A}
For any prime ideal $\p\subseteq \OK$, possibly ramified, let $\P$ be any prime ideal dividing
$\p\mathcal{O}_\L$, let $I$ be the inertia group of $\P$ and $\tau$ be one of the Frobenius automorphisms
corresponding to $\P/\p$. Let
\[
\theta(C;\p^m)  := \frac{1}{|I|}\sum_{a\in I}\eps_C(\tau^m a).
\]
Notice that $\theta(C;\p^m)\in[0,1]$, and that $\theta(C;\p^m)=\eps_C(\Artin{\L/\K}{\p^m})$ for
any non-ramified prime $\p$ and power $m$. Thus $\theta(C;\cdot)$ extends $\eps_C(\Artin{\L/\K}{\cdot})$
to ramifying prime ideals powers. %
With $\theta(C;\cdot)$ at our disposal we define the new function
\[
\psi(C;x)       := \sum_{\substack{\iI\subset\OK\\\Norm\iI\leq x}}\theta(C;\iI)\Lambda_\K(\iI).
\]
Observe that $\psi_C(x)$ and $\psi(C;x)$ are essentially equivalent since they agree except on
ramified-prime-powers ideals. However, $\psi(C;\cdot)$ is easier to deal with, since
$\theta(C;\cdot)$ is well defined for every prime ideal.\\
We further set
\[
\psi^{(1)}(C;x) := \int_{0}^{x}\psi(C;t)\dd t
\]
and, for $s>1$,
\[
K(C;s) := \sum_{\iI\subseteq \OK}\theta(C;\iI)\Lambda_\K(\iI)(\Norm\iI)^{-s}.
\]
As in~\cite[Ch.~IV Sec.~4, p. 73]{Ingham2} and~\cite[Sec.~5]{LagariasOdlyzko}, we have the integral
representation
\begin{equation}\label{eq:2.1A}
\psi^{(1)}(C;x) = \frac{1}{2\pi i}\int_{2-i\infty}^{2+i\infty} K(C;s)\frac{x^{s+1}}{s(s+1)}\,\dd s.
\end{equation}
The function $\theta(C;\cdot)$ is a class function and therefore can be written as a linear combination
of characters of irreducible representations of the group $G$. A clever trick (due to
Deuring~\cite{Deuring1} and MacCluer~\cite{MacCluer}, see also Lagarias and
Odlyzko~\cite[Lemma~4.1]{LagariasOdlyzko} and~\cite[p.~445--446]{GrenieMolteni9}) allows to write this
function as a linear combination of characters which are induced from characters of a certain cyclic
subgroup $H$ of $G$ specified below. Namely,
\begin{equation}\label{eq:2.2A}
K(C;s)
= -\frac{|C|}{|G|}\sum_\chi \bar\chi(g) \frac{L'}{L}(s,\chi,\L/\E),
\end{equation}
where $g$ is any fixed element in $C$, $\E:= \L^H$ is the subfield of $\L$ fixed by $H := \langle
g\rangle$, $L(s,\chi,\L/\E)$ is the Artin $L$-function associated with the extension $\L/\E$ and the
character $\chi$, and the sum runs on all irreducible characters $\chi$ of $H$. Since the extension is
abelian, this coincides with a suitable Hecke $L$-function, by class field theory.\\
With~\eqref{eq:2.1A},~\eqref{eq:2.2A} produces the identity
\begin{equation}\label{eq:2.3A}
\frac{|G|}{|C|}\psi^{(1)}(C;x)
= -\sum_\chi \bar\chi(g)\frac{1}{2\pi i}\int_{2-i\infty}^{2+i\infty} \frac{L'}{L}(s,\chi,\L/\E)\frac{x^{s+1}}{s(s+1)}\,\dd s.
\end{equation}
Finally, we introduce a special notation for the type of sum on characters as the one appearing
in~\eqref{eq:2.3A}, and for any $f\colon \widehat{\Gal(\L/\E)}\to \C$ we set
\[
\MC f := \sum_\chi \bar\chi(g)f(\chi).
\]
With this language, Equality~\eqref{eq:2.3A} reads
\begin{equation}\label{eq:2.4A}
\frac{|G|}{|C|}\psi^{(1)}(C;x)
= \MC I_\chi(x),
\end{equation}
where
\begin{equation}\label{eq:2.5A}
I_\chi(x)
:= -\frac{1}{2\pi i}\int_{2-i\infty}^{2+i\infty} \frac{L'}{L}(s,\chi,\L/\E)\frac{x^{s+1}}{s(s+1)}\,\dd s.
\end{equation}

\section{Some computations with Abelian Artin $L$-functions}\label{sec:3A}
Let $\E\subseteq \L$ be an abelian extension of fields and let $\chi$ be any irreducible character
of $\Gal(\L/\E)$. We will use $L(s,\chi)$ to denote $L(s,\chi,\L/\E)$. Also, set $\delta_\chi = 1$ if
$\chi$ is the trivial character, and $0$ otherwise.

We recall that for each $\chi$ there exist non-negative integers
$a_\chi$, $b_\chi$ such that
\[
a_\chi + b_\chi = \nE
\]
and a positive integer $Q(\chi)$ such that if we define
\begin{equation}\label{eq:3.1A}
\Gamma_\chi(s)
:= \Big[\pi^{-\frac{s}{2}}\Gamma\Big(\frac{s}{2}\Big)\Big]^{a_\chi}
   \Big[\pi^{-\frac{s+1}{2}} \Gamma\Big(\frac{s+1}{2}\Big)\Big]^{b_\chi}
\end{equation}
and
\begin{equation}\label{eq:3.2A}
\xi(s,\chi) := [s(s-1)]^{\delta_\chi} Q(\chi)^{s/2}\Gamma_\chi(s)L(s,\chi),
\end{equation}
then $\xi(s,\chi)$ satisfies the functional equation
\begin{equation}\label{eq:3.3A}
\xi(1-s,\bar\chi) = W(\chi)\xi(s,\chi),
\end{equation}
where $W(\chi)$ is a certain constant of absolute value $1$. Furthermore, $\xi(s,\chi)$ is an entire
function (by class field theory) of order $1$ and does not vanish at $s = 0$, and hence by Hadamard's
product theorem we have
\begin{equation}\label{eq:3.4A}
\xi(s,\chi) = e^{A_\chi+B_\chi s} \prod_{\rho\in Z_\chi} \Big(1 - \frac s\rho\Big) e^{s/\rho}
\end{equation}
for some constants $A(\chi)$ and $B(\chi)$, where $Z_\chi$ is the set of zeros (multiplicity included) of
$\xi(s,\chi)$. They are precisely those zeros $\rho = \beta + i\gamma$ of $L(s,\chi)$ for which $0 <
\beta < 1$, the so-called ``non-trivial zeros'' of $L(s,\chi)$. From now on $\rho$ will denote a
non-trivial zero of $L(s,\chi)$.

Differentiating~\eqref{eq:3.2A} and~\eqref{eq:3.4A} logarithmically we obtain the identity
\begin{equation}\label{eq:3.5A}
\frac{L'}{L}(s,\chi)
  = B_\chi
   + \sum_{\rho} \Big(\frac{1}{s-\rho}+\frac{1}{\rho}\Big)
   - \frac{1}{2}\log Q(\chi)
   - \delta_\chi\Big(\frac{1}{s}+\frac{1}{s-1}\Big)
   - \frac{\Gamma'_\chi}{\Gamma_\chi}(s),
\end{equation}
valid identically in the complex variable $s$.\\
Using~\eqref{eq:3.2A}, \eqref{eq:3.3A} and~\eqref{eq:3.5A} one sees that
\begin{equation}\label{eq:3.6A}
\begin{array}{ll}
\displaystyle\frac{L'}{L}(s,\chi) = \frac{a_\chi-\delta_\chi}{s} + r_\chi  + O(s)    & \text{as $s\to 0 $},\\[.4cm]
\displaystyle\frac{L'}{L}(s,\chi) = \frac{b_\chi}{s+1}           + r'_\chi + O(s+1)  & \text{as $s\to -1$},
\end{array}
\end{equation}
where
\begin{align}
\displaystyle
r_\chi  &= B_\chi
         + \delta_\chi
         - \frac{1}{2} \log\frac{Q(\chi)}{\pi^{\nE}}
         - \frac{a_\chi}{2}\frac{\Gamma'}\Gamma(1)
         - \frac{b_\chi}{2}\frac{\Gamma'}\Gamma\Big(\frac{1}{2}\Big),  \label{eq:3.7A}\\[.4cm]
\displaystyle
r'_\chi &=- \frac{L'}{L}(2,\bar{\chi})
         - \log\frac{Q(\chi)}{\pi^{\nE}}
         - \frac{\nE}{2}\frac{\Gamma'}\Gamma\Big(\frac{3}{2}\Big)
         - \frac{\nE}{2}\frac{\Gamma'}\Gamma(1).                       \label{eq:3.8A}
\end{align}
Comparing~\eqref{eq:3.7A} and~\eqref{eq:3.5A} with $s=2$, we further get
\begin{equation}\label{eq:3.9A}
r_\chi = \frac{L'}{L}(2,\chi)
        - \sum_{\rho}\frac{2}{\rho(2-\rho)}
        + \frac{5}{2}\delta_\chi
        + b_\chi.
\end{equation}

Shifting the axis of integration in~\eqref{eq:2.5A} arbitrarily far to the left, we collect the terms
coming from the pole of $L$ at $s=1$ (if any), the non-trivial zeros, the pole of the kernel (and of
$L'/L$, if any) at $s=0$, the pole of the kernel (and of $L'/L$, if any) at $s=-1$ and all the remaining
terms coming from the trivial zeros of $L$. This procedure gives the identity
\begin{equation}\label{eq:3.10A}
I_\chi(x) = \delta_\chi\frac{x^2}{2}
          - \sum_{\rho\in Z_\chi}\frac{x^{\rho+1}}{\rho(\rho+1)}
          - x r_\chi
          + r'_\chi
          + R_\chi(x)
\qquad
\forall x > 1,
\end{equation}
where $r_\chi$ and $r'_\chi$ are defined in~\eqref{eq:3.6A} and $R_\chi(x)$ is the explicit function
\begin{align*}
f_1(x)    &:= \sum_{r=1}^{\infty}\frac{x^{1-2r}}{2r(2r-1)},
\qquad
f_2(x)     := \sum_{r=2}^{\infty}\frac{x^{2-2r}}{(2r-1)(2r-2)},     \\
R_\chi(x) &:= - (a_\chi-\delta_\chi)(x\log x-x)
              + b_\chi    (\log x + 1)
              - a_\chi     f_1(x)
              - b_\chi     f_2(x).
\end{align*}
(with $x > 1$). The correctness of this procedure is proved in a way similar
to~\cite[\S~6]{LagariasOdlyzko}, further simplified by the fact that the integral is absolutely
convergent on vertical lines (see also~\cite[Ch.~IV Sec.~4, p. 73]{Ingham2}).\\
According to~\eqref{eq:2.4A}, in order to proceed we need to know the effect of the $\MC$ operator on
each term in~\eqref{eq:3.10A}. To this effect, we recall a few lemmas that we will need in the following.
\begin{lemma}[{\!\!\cite[Lemma~1]{GrenieMolteni9}}]\label{lem:3.1A}
Let
\[
\Sg:=
\begin{cases}
r_1(\L) +  r_2(\L)   & \text{if $g$ has order 1}, \\
r_2(\L) - 2r_2(\E)   & \text{if $g$ has order 2}, \\
0                    & \text{otherwise}.
\end{cases}
\]
Moreover let $\delta_C$ be defined to be $1$ if $C$ is the trivial class and $0$ otherwise. Then
\begin{align*}
\MC a_\chi &= \sum_{\chi}\bar{\chi}(g) a_\chi = \Sg,                               \\
\MC b_\chi &= \sum_{\chi}\bar{\chi}(g) b_\chi = \delta_C\nE-\Sg = \delta_C\nL-\Sg.
\end{align*}
\end{lemma}
From now on, we assume that $\L/\E$ is cyclic, and let $Z$ be the multiset of zeros of the Dedekind zeta
function $\zeta_\L$. Thus $Z$ is the disjoint union of the sets $Z_\chi$ for
$\chi\in\widehat{\Gal(\L/\E)}$.
\begin{lemma}[{\!\!\cite[Lemma~2]{GrenieMolteni9}}]\label{lem:3.2A}
Let $f$ be any complex function with $\sum_{\rho\in Z} |f(\rho)|<\infty$. Then
\[
\MC\sum_{\rho\in Z_{\chi}} f(\rho)
 = \sum_{\rho\in Z} \epsilon(\rho) f(\rho)
\]
where, for any $\rho\in Z$, $|\epsilon(\rho)|=1$ and $\epsilon(\overline{\rho}) =
\overline{\epsilon(\rho)}$.
\end{lemma}
The following lemma comes from~\eqref{eq:3.9A} and Lemmas~\ref{lem:3.1A} and~\ref{lem:3.2A}.
\begin{lemma}[{\!\!\cite[Lemma~3]{GrenieMolteni9}}]\label{lem:3.3A}
\[
\MC r_\chi = 2\sum_{\rho\in Z} \frac{\epsilon(\rho)}{\rho(2-\rho)}
             - \frac{\nL}{\nK|C|}\sum_{\iI\subseteq \OK}\theta(C;\iI)\frac{\Lambda_\K(\iI)}{(\Norm\iI)^{2}}
             + \nL\delta_C
             - \Sg
             + \frac{5}{2}.
\]
\end{lemma}
\begin{lemma}[{\!\!\cite[Lemma~5]{GrenieMolteni9}}]\label{lem:3.4A}
Define for any $x>1$, $R_C(x):=\MC R_\chi(x)$. Then
\[
R_C(x) =\int_0^x\log u\dd u - \Sg\int_1^{x+1}\log u\dd u
         + \delta_C\frac{\nL}{2}\Big[\log(x^2-1) + x \log\Big(\frac{x+1}{x-1}\Big)\Big].
\]
\end{lemma}
\begin{lemma}[{\!\!\cite[Lemma~10]{GrenieMolteni9}}]\label{lem:3.5A}
Assume GRH. Then
\[
\sum_{\rho \in Z} \frac{1}{|\rho(\rho+1)|} \leq 0.5375\lDL - 1.0355\nL + 5.3879 - 0.2635 r_1(\L).
\]
\end{lemma}
We finally prove three technical lemmas.
\begin{lemma}\label{lem:3.6A}
Assume GRH. Then
\[
\MC r_\chi \leq 1.075\lDL - 1.571\nL + 13.276.
\]
\end{lemma}

\begin{proof}
By Lemma~\ref{lem:3.3A}, we have
\[
\MC r_\chi \leq 2\sum_{\rho\in Z} \frac{1}{|\rho(2-\rho)|}
             + \nL\delta_C
             - \Sg
             + \frac{5}{2}.
\]
A brief check shows that $\nL\delta_C - \Sg \leq r_2(\L) \leq \frac{1}{2}\nL$. Moreover,
$|\rho(2-\rho)|=|\rho(\rho+1)|$, thus Lemma~\ref{lem:3.5A} applies here and the result
follows.
\end{proof}
\begin{lemma}\label{lem:3.7A}
We have
\[
-\MC r'_\chi \leq \lDL.
\]
\end{lemma}
\begin{proof}
As a consequence of~\eqref{eq:3.8A}, we have
\[
-\MC r'_\chi
 = \MC \frac{L'}{L}(2,\bar{\chi})
 + \MC \log Q(\chi)
 - \nE\Big(\log \pi - \frac{1}{2}\frac{\Gamma'}{\Gamma}\Big(\frac{3}{2}\Big) - \frac{1}{2}\frac{\Gamma'}{\Gamma}(1)\Big)\MC 1.
\]
Letting $C_1$ to be the class of $g^{-1}$, we see from~\eqref{eq:2.2A} that
\[
\MC \frac{L'}{L}(2,\bar{\chi}) = -\frac{|G|}{|C|} \overline{K(C_1;2)}
\]
which, by definition of $K$, is a negative real. Moreover,
\[
| \MC \log Q(\chi)|
= |\sum_\chi \bar{\chi}(g)\log Q(\chi)|
\leq \sum_\chi \log Q(\chi)
=    \lDL,
\]
by the product formula for conductors. The result follows because $\nE\MC 1=\nL\delta_C\geq 0$ and $\log
\pi - \frac{1}{2}\frac{\Gamma'}{\Gamma}\big(\frac{3}{2}\big) - \frac{1}{2}\frac{\Gamma'}{\Gamma}(1) =
1.41\ldots$ is positive.
\end{proof}
\begin{lemma}\label{lem:3.8A}
If $\L\neq \Q$, for any $x>1$,
\[
-R_C(x) \leq (\nL-1)\int_1^{x+1}\log u\dd u.
\]
\end{lemma}
\begin{proof}
Consider the formula for $R_C(x)$ given in Lemma~\ref{lem:3.4A}. %
When $r_2(\L)\geq 1$ we have $\Sg\leq r_1(\L)+r_2(\L) = \nL-r_2(\L)\leq \nL-1$ producing
\begin{align*}
-R_C(x) &\leq\Sg\int_1^{x+1}\log u\dd u
         \leq(\nL-1)\int_1^{x+1}\log u\dd u.
\end{align*}
On the other hand, if $r_2(\L)=0$ then when $\delta_C=0$ we have $\Sg=0$ and $R_C(x)>0$, while if
$\delta_C=1$ we have $\Sg=r_1(\L)=\nL$, $\frac{\nL}{2}\delta_C \geq 1$ because $\L\neq \Q$, and
\begin{align*}
-R_C(x) &\leq\nL\int_1^{x+1}\log u\dd u
             - \int_0^x\log u\dd u
             - \Big[\log(x^2-1) + x \log\Big(\frac{x+1}{x-1}\Big)\Big]   \\
        &=   (\nL-1)\int_1^{x+1}\log u\dd u
             - \log(x-1) - x\log\Big(\frac{x}{x-1}\Big)
         \leq(\nL-1)\int_1^{x+1}\log u\dd u.
\qedhere
\end{align*}
\end{proof}

\section{Proof of the theorem}\label{sec:4A}
When $\L=\Q$ the claim follows easily by Chebyshev's bound $\pi(x)\geq \frac{x}{2\log x}$. For the next
computations we assume $\L\neq\Q$.
\begin{lemma}\label{lem:4.1A}
Let $x\geq 400$ and $y>0$, then
\[
(x-y)\log y \leq x(\log x - \log(2\log x)).
\]
\end{lemma}
\begin{proof}
Let $f_x(y):=(x-y)\log y$. Its maximum is attained at a unique point $y_0(x)\in(1,x)$, with $y_0(\log
y_0+1)=x$. The formula shows that $y_0$ grows as a function of $x$. A simple computation shows that
\begin{align*}
\frac{f_x(y_0)}{x} - \log x + \log\log x
= g(1+\log y_0),
\end{align*}
where
\[
g(z) := \log\Big(1 + \frac{\log z - 1}{z}\Big) + \frac{1}{z} - 1.
\]
This function decreases for $z\geq e$
and is lower than $-\log 2$ when $z\geq 5.3193$. Since $5.3193e^{4.3193}$ $= 399.67\ldots$, the claim is
proved.
\end{proof}
Let $a_C(n) := \sharp\{\p\colon \p\text{ unramified }, \Artin{\L/\K}{\p} = C,\ \Norm\p = n\}$ and let also
\[
\vartheta_C(x)
:= \sum_{\substack{\p\\\p\text{ non-ram.}\\\Norm\p\leq x}}\eps_C\Big(\Artin{\L/\K}{\p}\Big)\log\Norm\p
 = \sum_{n\leq x}a_C(n)\log n,
\qquad
\vartheta_C^{(1)}(x) := \int_{0}^{x}\vartheta_C(t)\dd t.
\]
Then, by Lemma~\ref{lem:4.1A}, for $x\geq 400$,
\begin{align}
\vartheta_C^{(1)}(x)
&=    \sum_{n\leq x}a_C(n)(x-n)\log n
\leq x(\log x-\log(2\log x))\sum_{n\leq x}a_C(n)         \notag\\
&=    \pi_C(x)x(\log x-\log(2\log x)).                   \label{eq:4.1A}
\end{align}
Now we produce a lower bound for $\vartheta_C^{(1)}(x)$ out of a lower bound for $\psi^{(1)}(C;x)$.\\
To ease the notation we set $g_c:=|G|/|C|$ and observe that this is a positive integer.\\
By~\eqref{eq:2.4A},~\eqref{eq:3.10A} and Lemma~\ref{lem:3.2A}, we get
\[
g_c\psi^{(1)}(C;x)
 = \MC I_\chi(x)
 = \frac{x^2}{2}
  - \sum_{\rho\in Z}\epsilon(\rho)\frac{x^{\rho+1}}{\rho(\rho+1)}
  - x \MC r_\chi
  + \MC r'_\chi
  + R_C(x)
\]
which with the GRH assumption yields
\[
\frac{x^2}{2} - g_c\psi^{(1)}(C;x)
\leq x^{3/2}\sum_{\rho\in Z}\frac{1}{|\rho(\rho+1)|}
     + x\MC r_\chi
     - \MC r'_\chi
     - R_C(x).
\]
With Lemmas~\ref{lem:3.5A}--\ref{lem:3.8A}, this gives
\begin{align*}
\frac{x^2}{2} - &g_c\psi^{(1)}(C;x)
\leq (0.5375(x^{3/2}+2x) + 1)\lDL
     + x^{3/2}(5.4 - 1.0355\nL)    \\
&    + \Big(\int_1^{x+1}\log u\dd u
     - 1.571x\Big) \nL
     + 13.276 x
     - \int_1^{x+1}\log u\dd u.
\end{align*}
When $x\geq 400$ the term in $\nL$ appearing in the last line is bounded by $\nL x(\log x - 2.55)$ and
the sum of the last two terms by $8.3 x$.
%
%
%
We further simplify their contribution noticing that
\[
\nL x(\log x - 2.55) + 8.3 x
=    \nL x(\log x - 2.4) - 0.15 \nL x + 8.3 x
\leq \nL x\log x - 2.4 x \nL + 8 x,
\]
where in the last step we used that $\nL \geq 2$. We thus have
\begin{align*}
\frac{x^2}{2} - g_c\psi^{(1)}(C;x)
\leq& (0.5375(x^{3/2}+2x)+1)\lDL
     + x^{3/2}(5.4 - 1.0355\nL)                    \\
    &+ \nL x \log x
     - 2.4 \nL x
     + 8 x.
\end{align*}
Now we remove the contribution to $\psi^{(1)}(C;x)$ of the prime powers $\p^m$ with $m\geq 2$. Let
\[
\vartheta(C;x)       := \sum_{\substack{\p\subset\OK\\\Norm\p\leq x}}\theta(C;\p)\log(\Norm\p),
\qquad
\vartheta^{(1)}(C;x) := \int_{0}^{x}\vartheta(C;t)\dd t.
\]
The estimation in~\cite[Th.~13]{RosserSchoenfeld} gives $0\leq \psi^{(1)}(C;x) - \vartheta^{(1)}(C;x)
\leq 1.43 \frac{2}{3}x^{3/2}\nK$. Thus
\begin{align*}
\frac{x^2}{2} - g_c\vartheta^{(1)}(C;x)
\leq& (0.5375 x^{3/2} + 1.075 x + 1)\lDL \\
    &+ x^{3/2}(5.4 - 0.082\nL)
     + \nL x \log x
     - 2.4 \nL x
     + 8 x
\end{align*}
%
which simplifies to
\begin{equation}\label{eq:4.2A}
\frac{x^2}{2} - g_c\vartheta^{(1)}(C;x)
\leq (0.5375 x^{3/2} + 1.075 x + 1)\lDL
     + 2 \nL x + 5.4 x^{3/2} + 8 x,
\end{equation}
because $(-0.082x^{3/2} + x\log x - 2.4x)/x$ has a maximum at $x = (2/0.082)^2 = 594.88\ldots$ where it is
lower that $2$.
%
The quantities $\vartheta(C;x)$ and $\vartheta_C(x)$ differ only by the contribution of the ramified
prime ideals to $\vartheta(C;x)$. In fact,
\[
0\leq \vartheta(C;x) - \vartheta_C(x)
\leq \sum_{\substack{\p\\\p\text{ ram.}\\\Norm\p\leq x}}\log\Norm\p
\leq \sum_{\p\text{ ram.}}\log\Norm\p
\leq \log(\Norm\DLK)
\leq \lDL.
\]
Hence,
\[
0\leq \vartheta^{(1)}(C;x) - \vartheta_C^{(1)}(x)
\leq (x-1) \lDL,
\]
which with~\eqref{eq:4.2A} gives
\begin{equation}\label{eq:4.3A}
\frac{x^2}{2} - g_c\vartheta_C^{(1)}(x)
\leq (0.5375 x^{3/2} + g_c x + 1.075 x)\lDL
     + 2 \nL x + 5.4 x^{3/2} + 8 x.
\end{equation}
By~\eqref{eq:4.1A} and~\eqref{eq:4.3A}, in order to have $\pi_C(x)>k$ it is sufficient to have
\[
\frac{x^2}{2}
> (0.5375x^{3/2} + g_c x + 1.075x)\lDL
  + 2 \nL x
  + 5.4 x^{3/2}
  + 8 x
  + k g_c x(\log x-\log(2\log x)),
\]
i.e.
\begin{equation}\label{eq:4.4A}
\sqrt{x}
> \Big(1.075 + \frac{2g_c + 2.15}{\sqrt{x}}\Big)\lDL
  + 4 \frac{\nL}{\sqrt{x}}
  + 10.8
  + \frac{16}{\sqrt{x}}
  + 2k g_c\frac{\log x-\log(2\log x)}{\sqrt{x}}
\end{equation}
which is true when
\[
\sqrt{x} = 1.075\lDL + \sqrt{2g_c k\log(g_c k)}  + 2g_c + 15.
\]
\begin{proof}
Let
\[
A:=1.075\lDL + 2g_c + 15
\]
and
\[
B:=\Big(1.075 + \frac{2g_c + 2.15}{\sqrt{x}}\Big)\lDL
   + 4\frac{\nL}{\sqrt{x}}
   + 10.8
   + \frac{16}{\sqrt{x}}.
\]
To show that~\eqref{eq:4.4A} holds with the indicated value of $x$, it is sufficient to prove
\begin{equation}\label{eq:4.5A}
A + \sqrt{2g_ck\log(g_ck)} > B + 2k g_c\frac{\log x-\log(2\log x)}{\sqrt{x}}.
\end{equation}
We have
\[
  (A - B - 1)\sqrt{x}
  = (2g_c+3.2)\sqrt{x} - ((2g_c+2.15)\lDL + 4\nL + 16)
  \geq 1.4\lDL - 4\nL + 72
\]
%
which is positive, according to entry~$b=4$ in~\cite[Table~3]{OdlyzkoTables}.
%
%
Since $A - B > 1$, our claim will hold if
\[
 \sqrt{2g_c k\log(g_c k)} + 1
\geq
2g_c k\frac{\log x - \log(2\log x)}{\sqrt{x}},
\]
i.e. $k=0$ or
\begin{equation}\label{eq:4.6}
\frac{\sqrt{\log y}}{\sqrt{2y}} + \frac{1}{2y}
\geq
\frac{\log x - \log(2\log x)}{\sqrt{x}},
\end{equation}
where $y:=g_ck \geq 1$. The right-hand side decreases if $x\geq 30$
hence is at most $0.2$
and the left-hand side is larger than $0.2$ for $1\leq y\leq 120$.
We thus assume $y\geq 120$, and in that case $x\geq 2y\log y\geq 30$,
hence~\eqref{eq:4.6} holds if
\[
\frac{\sqrt{\log y}}{\sqrt{2y}}\geq \frac{\log(2y\log y)-\log(2\log(2y\log y))}{\sqrt{2y\log y}}
\]
i.e.
\[
\log y\geq \log(2y\log y)-\log(2\log(2y\log y))
\]
which is obviously true in this range.
\end{proof}
\noindent
This proves the claim under the assumption that $x\geq 400$. The exceptions to this condition are the
cases where
\[
  1.075\lDL + \sqrt{2g_c k\log(g_c k)} + 2g_c + 15 < 20
\]
and this happens only when $g_c = 1$, $\DL\leq 16$ and $k \leq 2$.\\
%
%
For these remaining cases we check directly the existence of the corresponding ideals. We observe that
$g_c=|G|/|C|=1$ if and only if $|G|=1$ and hence $\L=\K$. Moreover, $\DL\leq 16$ implies $\nL\leq 2$.
Hence it is sufficient to check that, in quadratic fields, there are at least three ideals of norm at
most $\intpart{(1.075\log 3 + 17)^2}=330$.
They exist because the primes above $2$, $3$ and $5$ have norm at most $25$.

\bibliographystyle{amsplain}

\end{document}